\documentclass[11pt,reqno]{amsart}

\usepackage{url}
\usepackage{amssymb}
\usepackage{amscd}
\usepackage{amsfonts}
\usepackage{amsmath}
\usepackage{eepic}
\usepackage{gastex}
\usepackage{amsthm}
\usepackage{amsgen}
\usepackage{eucal}
\usepackage[square]{natbib}

\theoremstyle{plain}
\newtheorem{theorem}{Theorem}

\newtheorem{corollary}[theorem]{Corollary}

\theoremstyle{remark}
\newtheorem{remark}{Remark}

\newcommand{\sgp}{semi\-group}
\newcommand{\sgps}{semi\-groups}

\DeclareMathOperator{\Id}{\mathrm{Id}}

\def\malcev{\mathop{\hbox{$\bigcirc$\kern-9.5pt\raise1pt\hbox{\scriptsize$m$}\kern1.5pt}}}

\title[The identities of free products of two-element monoids]{The identities of the free product\\ of a pair of two-element monoids}

\author{M. V. Volkov}

\address{Institute of Natural Sciences and Mathematics, Ural Federal University,
Lenina 51, 620000 Ekaterinburg, Russia} \email{m.v.volkov@urfu.ru}

\begin{document}

\begin{abstract}
Up to isomorphism, there exist two non-isomorphic two-el\-ement monoids. We show that the identities of the free product of every pair of such monoids admit no finite basis.
\end{abstract}

\maketitle

Let $S_2$ stand for the \sgp\ defined by the \sgp\ presentation $\langle e,f \mid e^2=e,\ f^2=f\rangle$; in other words, $S_2$ is the free product of two one-element semigroups (in the category of semigroups). It is known (and easy to verify) that $S_2$ is the only free product in the category of semigroups that satisfies a nontrivial identity~\citep{Sh72}.
Shneerson and the author \citeyearpar{SV17} have characterized the identities of $S_2$ and proved that these identities admit a finite basis.

In the present note we address the identities of free products in the category of monoids considered as algebras of type (2,0). If $M_1$ and $M_2$ are two monoids, then when constructing their monoidal free product $M_1*M_2$, one has to amalgamate the identity elements of $M_1$ and $M_2$. Therefore the free monoidal product of two one-element monoids is again the one-el\-ement monoid. Moreover, it is evident that the product $M_1*M_2$ is isomorphic to one of its factors whenever the other factor is the one-element monoid. In view of this observation, if we want the operation of free product to produce something new, we have to assume that both $M_1$ and $M_2$ contain at least two elements. On the other hand, if one of the factors $M_1$ of $M_2$ contains at least three elements, $M_1*M_2$ cannot satisfy any nontrivial identity. Indeed, let $|M_1|\ge2$ and $|M_2|\ge3$, say. Take $a\in M_1\setminus\{1\}$ and let $b,c\in M_2\setminus\{1\}$ be such that $b\ne c$. Then it is easy to see that the elements $ab$ and $ac$ generate a free sub\sgp\ in the free product $M_1*M_2$.

Thus, studying identities of the free product $M_1*M_2$ makes sense only if both $M_1$ and $M_2$ consist of two elements. Up to isomorphism, there exist two non-isomorphic two-element monoids: one is the two-element idempotent monoid, which we denote by $\mathbb{I}_2$, and the other one is the two-element group, which we denote by $\mathbb{C}_2$.
Therefore, up to isomorphism, the free products to be considered are $\mathbb{I}_2*\mathbb{I}_2$, $\mathbb{C}_2*\mathbb{C}_2$, and $\mathbb{I}_2*\mathbb{C}_2$. These free products can be defined by the following monoid presentations:
\begin{gather}
\mathbb{I}_2*\mathbb{I}_2=\langle e,f \mid e^2=e,\ f^2=f\rangle,\label{eq:1}\\
\mathbb{C}_2*\mathbb{C}_2=\langle a,b \mid a^2=1,\ b^2=1\rangle,\label{eq:2}\\
\mathbb{I}_2*\mathbb{C}_2=\langle e,b \mid e^2=e,\ b^2=1\rangle.\label{eq:3}
\end{gather}

The monoid defined by the presentation \eqref{eq:1} is denoted $J_\infty$. Observe that the monoid presentation \eqref{eq:1} looks exactly as the semigroup presentation used above to define the semigroup $S_2$, whence the monoid $J_\infty$  is nothing but the semigroup $S_2$ with identity element adjoined. Shneerson and the author \citeyearpar{SV17} have shown that the identities of $S_2$ with identity element adjoined are not finitely based.

It is easy to see that the presentation \eqref{eq:2} defines a group known in the literature as the \emph{infinite dihedral group} $D_\infty$. The group $D_\infty$ is an extension of the infinite cyclic group generated by the element $ba$ by the two-element group; in particular, $D_\infty$ is a metabeian group. A general result by \citet{Sa87} implies that if the semigroup identities of a group $G$ are finitely based and $G$ is an extension of an abelian subgroup by a group of finite exponent, then $G$ either is abelian or has finite exponent; see \citep[Proposition~6]{Sa87}. Applying this result to $D_\infty$, we conclude that the semigroup identities of $D_\infty$ are not finitely based. Observe that, in contrast, the \textbf{group} identities of $D_\infty$ are finitely based; this follows from another general result, due to \citet{Co67}, who proved that the group identities of any metabelian group admit a finite basis.

Thus, it remains to analyze the identities of the monoid defined by the presentation \eqref{eq:3}. This monoid is generated by an idempotent and an involution, and we denote it by $K_\infty$, having in mind Kuratowski's closure-complement theorem\footnote{The classic version of Kuratowski's closure-complement theorem basically describes the monoid generated by two operators on subsets of an arbitrary topological space: the operator of taking the closure of a subset and the operator of forming the complement of a subset; see the excellent survey by \cite{GJ08} for quite a comprehensive treatment. Many generalizations have been considered in which the operator of taking closure has been substituted by various idempotent operators while the operator of forming complement has been substituted by various involutive operators; some of these generalizations are surveyed in \citep[Subsection 4.2]{GJ08}. Clearly, all monoids of operators arising this way are homomorphic images of the monoid $K_\infty$.}. The main result of the present note is the following
\begin{theorem}
\label{thm:K_infty}
The identities of the monoid $K_\infty$ are not finitely based.
\end{theorem}

We prove Theorem~\ref{thm:K_infty} re-using the technique that was applied by Shneerson and the author \citeyearpar{SV17} to prove the analog of this theorem for the monoid $J_\infty$. (In fact, the same technique could have been applied to show the absence of a finite basis for the semigroup identities of the group $D_\infty$.) The technique stems from \cite{ACHLV15}; in order to present it, we need  to recall three concepts.

The first concept is that of Mal'cev product. The \emph{Mal'cev product} of two classes of \sgps\ $\mathbf{A}$ and $\mathbf{B}$, say, is the class $\mathbf{A}\malcev\mathbf{B}$ of all \sgps\ $S$ for which there exists a congruence $\theta$ such that the quotient \sgp\ $S/\theta$ lies in $\mathbf{B}$ while all $\theta$-classes that are sub\sgps\ in $S$ belong to $\mathbf{A}$. Notice that a $\theta$-class forms a subsemigroup of $S$ if and only if the class is an idempotent of the quotient $S/\theta$. We denote by $\mathbf{Com}$ and $\mathbf{Fin}$ the classes of all commutative \sgps\ and all finite semigroups, respectively.

The next concept we need is that of a Zimin word. Let $x_1,x_2,\dots,x_n,\dots$ be a sequence of letters. The sequence $\{Z_n\}_{n=1,2,\dots}$ of \emph{Zimin words} is defined inductively by
\[
Z_1(x_1):=x_1,\quad Z_{n+1}(x_1,\dots,x_{n+1}):=Z_n(x_1,\dots,x_n)x_{n+1}Z_n(x_1,\dots,x_n).
\]
Observe that in the word $Z_n$ the letter $x_i$, $i=1,\dots,n$, occurs $2^{n-i}$ times and the length of $Z_n$ is $2^n-1$.

Finally, recall a word $v$ is called an \emph{isoterm} relative to a semigroup $S$ if the only word $v'$ such that $S$ satisfies the identity $v\bumpeq v'$ is the word $v$ itself.

Now we state the main result of \cite{ACHLV15} in a form that it convenient for the use in the present note.

\begin{theorem}[{\mdseries\cite[Theorem~6]{ACHLV15}}]
\label{thm:main}
The identities of a \sgp\ $S$ have no finite basis provided that:
\begin{itemize}
\item[\emph{(i)}] $S$ lies in the Mal'cev product $\mathbf{Com}\malcev\mathbf{Fin}$, and
\item[\emph{(ii)}] each Zimin word is an isoterm relative to $S$.
\end{itemize}
\end{theorem}

\begin{proof}[Proof of Theorem~\ref{thm:K_infty}] 
From the definition of the free product, it readily follows that each non-identity element of the monoid $K_\infty$ can be uniquely represented as an alternating product of the generators $e$ and $b$.

First we show that $K_\infty$ satisfies the condition (i) of Theorem \ref{thm:main}. Consider the monoid $T$ defined by the following presentation:
\[
T=\langle f,g \mid f^2=fgf=f,\ g^2=1\rangle.
\]
It is easy to compute that $T$ consists of 6 elements: $1,g,f,fg,gf,gfg$, all of which except $g$ are idempotents. The map $e\mapsto f,\ b\mapsto g$ extends to a monoid homomorphism $K_\infty\to T$. The kernel $\theta$ of this homomorphism is a congruence on the monoid $K_\infty$ with two singleton classes $1\theta=\{1\}$ and $b\theta=\{b\}$ and four infinite classes:
\begin{align*}
e\theta&=\{(eb)^ke\mid k\ge0\}, &be\theta&=\{(be)^m\mid m>0\},\\
eb\theta&=\{(eb)^\ell\mid \ell>0\},& beb\theta&=\{(be)^nb\mid n>0\}.
\end{align*}
The infinite $\theta$-classes and the $\theta$-class $1\theta=\{1\}$ are subsemigroups in $K_\infty$. Clearly, the latter subsemigroup is commutative, and a direct computation shows that so are the four other subsemigroups. Thus, the monoid $K_\infty$ lies in the Mal'cev product $\mathbf{Com}\malcev\mathbf{Fin}$.

Now we aim to verify the condition (i) of Theorem \ref{thm:main}, that is, we want to show that no non-trivial identity of the form $Z_n\bumpeq z$ may hold in $K_\infty$. This verification repeats the argument used by Shneerson and the author \citeyearpar{SV17} for the monoid $J_\infty$, but we reproduce it for the reader's convenience. We induct on $n$. Observe that the subsemigroup $eb\theta=\{(eb)^\ell\mid\ell>0\}$ of $K_\infty$ is isomorphic to the additive semigroup of positive integers $\mathbb{N}$. It is well known that every identity $u\bumpeq v$ satisfied by $\mathbb{N}$ is \emph{balanced}, that is, every letter occurs the same number of times in $u$ and in $v$. Therefore if an identity of the form $Z_n\bumpeq z$ holds in $K_\infty$, it must be balanced, and this immediately implies that for $n=1$ any such identity must be trivial. Now assume that our claim holds for some $n$ and let a word $w=w(x_1,\dots,x_{n+1})$ be such that the identity $Z_{n+1}\bumpeq w$ holds in $K_\infty$. If we substitute 1 for $x_1$ in this identity, we conclude that also the identity
\[
Z_{n+1}(1,x_2,\dots,x_{n+1})\bumpeq w(1,x_2,\dots,x_{n+1})
\]
should hold in $k_\infty$. However, the word $Z_{n+1}(1,x_2,\dots,x_{n+1})$ is nothing but the Zimin word $Z_n(x_2,\dots,x_{n+1})$ and by the induction assumption we have $w(1,x_2,\dots,x_{n+1})=Z_n(x_2,\dots,x_{n+1})$. This, together with the fact that the identity $Z_{n+1}\bumpeq w$ must be balanced, means that the word $w(x_1,\dots,x_{n+1})$ is obtained from the Zimin word $Z_n(x_2,\dots,x_{n+1})$ by inserting $2^n$ occurrences of the letter $x_1$ in the latter. If the insertion is made in a way such that the occurrences of $x_1$ alternate with $2^n-1$ occurrences of $x_2,\dots,x_{n+1}$, then $w$ coincides with $Z_{n+1}$, and we are done. It remains to verify that any other way of inserting $2^n$ occurrences of $x_1$ in $Z_n(x_2,\dots,x_{n+1})$ produces a word $w$ such that the identity $Z_{n+1}\bumpeq w$ fails in $K_\infty$. Indeed, substitute $e$ for $x_1$ and $b$ for all other letters in this identity. The value of the left-hand side under this substitution is $(eb)^{2^n-1}e$. On the other hand, since at least two occurrences of $x_1$ in the word $w$ are adjacent, we are forced to apply the relation $e^2=e$ at least once to get a representation of the value of the right-hand side as an alternating product of the generators $e$ and $b$. Hence $e$ occurs less than $2^n$ times in this representation, and therefore, the value cannot be equal to $(eb)^{2^n-1}e$.

Theorem~\ref{thm:K_infty} now follows from Theorem~\ref{thm:main}. 
\end{proof}

Taking into account the discussion preceding the formulation of Theorem~\ref{thm:K_infty}, we arrive at the following  
\begin{corollary}
\label{cor:final}
For each pair of two-element monoids, the identities of their free product admit no finite basis.
\end{corollary}

\begin{remark}
\label{rem:D_infty}
We have already mentioned in passing that our technique could have been used to verify that the semigroup identities of the group $D_\infty$ admit no finite basis. Inspecting the above proof of Theorem~\ref{thm:K_infty}, the reader may see that this is indeed the case. The containment $D_\infty\in \mathbf{Com}\malcev\mathbf{Fin}$ immediately follows from the fact that $D_\infty$ is an extension of the infinite cyclic group by the two-element group. The inductive proof that all Zimin words are isoterms relative to $K_\infty$ works also for $D_\infty$ with the following minimal adjustments: one has to change the generator $e$ of $K_\infty$ to the generator $a$ of $D_\infty$ and use the relation $a^2=1$ from the presentation~\eqref{eq:2} instead of the relation $e^2=e$ from the presentation~\eqref{eq:3}.
\end{remark}

\begin{remark}
\label{rem:difference}
The similarity in the arguments used here and in \citep{SV17} to establish the absence of finite identity bases for each of the monoids $J_\infty$, $D_\infty$, and $K_\infty$ may lead one to the conjecture that some of the three monoids or perhaps all of them satisfy exactly the same identities. However, this is not the case. Indeed, the results of     \citep{SV17} easily imply that the monoid $J_\infty$ satisfies the identity $x^2yx\bumpeq xyx^2$. This identity fails in both $D_\infty$ and $K_\infty$ as one can see by evaluating the variables $x$ and $y$ at the generators of the latter monoids. Further, since the group $D_\infty$ is an extension of the infinite cyclic group by the two-element group, the identity $x^2y^2\bumpeq y^2x^2$ holds in $D_\infty$, but it fails in $J_\infty$ and $K_\infty$ as again is revealed by evaluating the variables at the generators.

The observations just made suffice to claim that the sets $\Id(J_\infty)$, $\Id(D_\infty)$, and $\Id(K_\infty)$ of identities holding in respectively $J_\infty$, $D_\infty$, and $K_\infty$ are pairwise distinct; moreover, the sets $\Id(J_\infty)$ and $\Id(D_\infty)$ are incomparable. In fact, one can show that $K_\infty$ satisfies the identity $x^4yx^2\bumpeq x^2yx^4$ that fails in $D_\infty$ so that the sets $\Id(D_\infty)$ and $\Id(K_\infty)$ are incomparable as well. One can also show that $\Id(J_\infty)$ is a proper subset of  $\Id(K_\infty)$. These results rely on a characterization of $\Id(K_\infty)$ which will be the subject of a separate paper.
\end{remark}

\bigskip

\small

\noindent\textbf{Acknowledgements.} The author acknowledges support from the Ministry of Education and Science of the Russian Federation, project no.\ 1.3253.2017, the Competitiveness Program of Ural Federal University, and from the Russian Foundation for Basic Research, project no.\ 17-01-00551.

\end{document}